\documentclass[12pt]{article}
\usepackage[utf8]{inputenc}
\usepackage{amssymb,amsmath,amsfonts,amsthm,amscd,latexsym,indentfirst,verbatim,xcolor}
\usepackage[T1]{fontenc}
\usepackage{geometry}
\usepackage{tikz}
\usetikzlibrary{arrows.meta,decorations.pathreplacing,calc}

\def\Mat{M}
\def\End{\operatorname{End}}
\def\Span{\textrm{Span}}
\def\supp{\operatorname{supp}}
\def\wt{\operatorname{wt}}
\def\rk{\operatorname{rk}}
\def\Th{\widehat T}
\def\Gh{\widehat\Gamma}
\def\tx{\mathtt{x}}
\def\ty{\mathtt{y}}
\def\ta{\mathtt{a}}

\theoremstyle{plain}
\newtheorem{theorem}{Theorem}[section]
\newtheorem{lemma}[theorem]{Lemma}
\newtheorem{proposition}[theorem]{Proposition}

\theoremstyle{definition}
\newtheorem{definition}[theorem]{Definition}

\begin{document}
\sloppy
\hfill{16N40, 16D30 (MSC2020)}
\begin{center}
{\Large A simple construction of a simple nil ring}
\smallskip

Vsevolod Gubarev
\end{center}

\begin{abstract}
In 1999, A. Smoktunowicz solved a long-standing problem and constructed a~simple nil ring.
We present a substantial simplification of this construction
found by Aristotle (Harmonic)~\cite{Aristotle} with the assistance of ChatGPT.
\end{abstract}

\section{Introduction}

At least in 1956, I. Kaplansky posed a problem of
existence of a simple nil ring as Problem~2 in his list of the twelve problems in ring theory~\cite{Kaplansky}.
In 1969, this problem was recorded as 1.172 in Dniester Notebook~\cite{Dnestrovskaya} by K. A. Zhevlakov.
In 1969, Yu.~M.~Ryabukhin found a useful criterion~\cite{Ryab}.
Finally, in 1999 (published in 2002~\cite{Smok}, see also the exposition~\cite{Lvov}) A.~Smoktunowicz solved this problem proving

\begin{theorem}\label{main}
Over every countable field $K$ there exists a simple associative nil algebra.
\end{theorem}

The main purpose of this paper is to give a short proof of Theorem~\ref{main}.
As in~\cite{Smok}, the starting point is Ryabukhin's criterion~\cite{Ryab}
(Proposition~\ref{ryab}): it suffices to construct a nil algebra $R$ containing an element $b\ne 0$ with $b\in Rb^2R$.
We take
$$
A = K\langle a,z_1,\ldots,z_5\rangle,\quad
 \omega = a-\sum_{j = 1}^5 z_ja^2z_j,
$$
and $R$ is a homomorphic image of $A$ in which $\omega$ vanishes and every element is nilpotent.

We construct finite-dimensional representations $T_i$ of $A$ in which the rank of
$T_i(a)$ is much larger than the rank of $T_i(\omega)$, while more and more elements
of $A$ are mapped to nilpotent matrices.
Section~\ref{s:matrices} shows that such a~sequence already yields the required algebra: we take the direct product
$E = \prod_i\Mat_{m_i}(K)$, factor out the ideal generated by $\Omega = (T_i(\omega))_i$, and let $R$ be the subalgebra generated by the images of the
six letters (Proposition~\ref{fromMatrices}).
Sequences with finite support vanish
in $E/(\Omega)$, which makes $R$ nil, while the rank condition keeps the image of
$a$ nonzero.

It remains to construct representations satisfying the requirements of Proposition~\ref{fromMatrices}.
We do this in three steps.
First, the representations are built from a complete binary tree (Section~\ref{s:tree}).
The letter~$a$ jumps over vertices, while the letters $z_i$ move along the tree with matrix weights.
The relation defining a chain makes the terms $z_i a^2z_i$ cancel all internal $a$-jumps, so that $T(\omega)$ survives only on the leaves.
This produces the required rank gap.
Second, for a~fixed polynomial~$f$ we cut the tree into bands between consecutive grid heights (Sections~\ref{s:passes} and~\ref{s:first}).
Every sufficiently long path has a~first crossing of one such band.
This reduces nilpotency of $T(f)$ to the vanishing of finitely many boundary blocks of a cut tree, and Theorem~\ref{nilpotence} makes this reduction precise.
Third, compression makes these boundary blocks vanish, while amplification increases the root dimension and preserves all conditions imposed earlier (Section~\ref{s:operations}).
Repeating the two operations lets us treat the elements of~$A$ one at a~time without losing the rank inequality (Section~\ref{s:construction}).

All algebras are associative and need not have a unit.
Given an algebra $S$, we denote by $S^\#  = K\oplus S$ the algebra with an adjoined unit.

\section{Preliminaries}\label{s:ryab}

For completeness, we provide the proof of Ryabukhin's criterion for the required direction.

\begin{proposition}[{\cite{Ryab}}]\label{ryab}
Let $R$ be a nil algebra over a~field~$K$ and let $0\ne b\in Rb^2R$.
Then some quotient algebra of~$R$ contains an ideal which is simple and nil.
\end{proposition}

\begin{proof}
The union of a chain of ideals of the $K$-algebra~$R$ not containing~$b$ is again such an ideal.
By Zorn's lemma, there is an ideal~$Q$ of~$R$ which is maximal among the ideals not containing~$b$.
Put $S = R/Q$.
Every nonzero ideal of~$S$ contains $\bar b$, since its preimage in~$R$ strictly contains~$Q$.
Hence $J = S^\# \bar bS^\# $ is the smallest nonzero ideal of $S$.
From $b\in Rb^2R$ we get $\bar b\in J^2$, so $J = J^2\neq (0)$.

Let $(0)\neq H\subseteq J$ be a two-sided ideal of~$J$ (not necessarily a~$K$-subspace).
The set $S^\# HS^\# $ is a nonzero ideal of the algebra $S$ contained in
$J$, so $S^\# HS^\#  = J$.
Therefore
$$
J = J^3
  = J(S^\# HS^\# )J
  \subseteq JHJ
  \subseteq H.
$$
Thus $H = J$, and $J$ is a simple ring.
It is a $K$-algebra and it is nil, since $S$ is nil.
\end{proof}

\section{Useful representations}\label{s:matrices}

Given a~sequence of positive integers $(m_i)_{i\geq1}$, consider the direct product
$E = \prod_{i\ge 1}\Mat_{m_i}(K)$.

\begin{lemma}\label{oneElement}
Let $\Omega = (\Omega_i)\in E$ and $r_i = \rk\Omega_i$.
a) For every $G = (G_i)\in(\Omega)$ there is a~constant $C$ such that $\rk G_i\le C r_i$ for all $i$.

b) If $\Omega_i\neq0$ for every $i$, then $(\Omega)$ contains every sequence with finite support.
\end{lemma}

\begin{proof}
(a) If $G = \sum_{s = 1}^C U_s\Omega V_s$, then
$\rk G_i\leq\sum_{s = 1}^C\rk(U_{s,i}\Omega_i V_{s,i})\leq Cr_i$.

(b) For every~$i$, the ideal generated by $\Omega_i$ equals $M_{m_i}(K)$,
hence~$(\Omega)$ contains all sequences with finite support.
\end{proof}

\begin{proposition}\label{fromMatrices}
Let $K$ be a countable field and let $f_1,f_2,\ldots$ be an enumeration of the
nonzero elements of $A$.
Suppose that there are representations
$T_i\colon A\to\Mat_{m_i}(K)$, integers $n_i\to\infty$ and positive integers $q_1,q_2,\ldots$ such that
\begin{equation}\label{hyp}
\rk T_k(a)>4^{n_k},\quad \rk T_k(\omega) = 2^{n_k}, \  k\geq1,\quad
 T_i(f_j)^{q_j} = 0, \ i\geq j\geq1.
\end{equation}
Then there is a nil algebra $R$ containing an element $b\neq0$ such that
$b = \sum_{j = 1}^5c_jb^2c_j$ for some $c_j\in R$.
In particular, $b\in Rb^2R$.
\end{proposition}

\begin{proof}
Let $\Phi\colon A\to E$ be the homomorphism $\Phi(h) = (T_i(h))_i$ and put $\Omega = \Phi(\omega)$.
Let $\pi\colon E\to E/(\Omega)$ be the quotient map, and
let $R$~be $\pi \Phi(A)$, i.\,e. the subalgebra of $E/(\Omega)$ generated by
$b = \pi\Phi(a)$ and $c_j = \pi\Phi(z_j)$, $1\leq j\leq 5$.
Since $\pi(\Omega) = 0$, we have $b = \sum_jc_jb^2c_j$.

By~\eqref{hyp},
$\rk T_i(a)/\rk T_i(\omega)>2^{n_i}\to\infty$.
By~Lemma~\ref{oneElement}(a), $\Phi(a)\notin(\Omega)$, thus, $b\neq0$.

Every element of~$R$ has the form $\pi\Phi(f)$ with $f\in A$. 
If $f\neq0$, then $f = f_j$ for some~$j$.
By~\eqref{hyp}, the sequence $\Phi(f_j)^{q_j} = (T_i(f_j)^{q_j})_i$ vanishes for $i\ge j$, so it has
finite support.
All $\Omega_i$ are nonzero, since $\rk\Omega_i = 2^{n_i}>0$.
By~Lemma~\ref{oneElement}(b), $\Phi(f_j)^{q_j}\in(\Omega)$, hence $(\pi\Phi(f_j))^{q_j} = 0$.
\end{proof}

The rest of the paper constructs representations satisfying~\eqref{hyp}.

\section{The tree and its matrices}\label{s:tree}

\subsection{Chains}
\begin{definition}
A \emph{chain} $S$ of height $N\ge 1$ consists of positive integers
$d_0,\ldots,d_N$ with $d_N = 1$ and matrices
$$
X_{j,i}\in\Mat_{d_j\times d_{j+1}}(K),\quad
Y_{j,i}\in\Mat_{d_{j+1}\times d_j}(K),\quad 0\leq j<N,\ 1\leq i\leq 5
$$
such that $\sum_{i = 1}^5X_{j,i}Y_{j,i} = I_{d_j}$ for every $j$.
The number $D = d_0$ is
the \emph{root dimension} of $S$.
\end{definition}

The equality $\sum_iX_{j,i}Y_{j,i}=I_{d_j}$ has the following goal: in Lemma 4.3 it makes the terms $z_i a^2z_i$ cancel the $a$-jump over every internal vertex.

The \emph{initial chain} has height~1, dimensions $(d_0,d_1) = (5,1)$, and
$X_{0,i} = e_i$ (a column), $Y_{0,i} = e_i^{\mathsf T}$ (a row), where $e_1,\ldots,e_5$ form the standard basis of~$K^5$.
%Since each product $X_{j,i}Y_{j,i}$ has rank at most $d_{j+1}$, we always have $d_j\le 5d_{j+1}$, so $D\le 5^N$.
We will use $D>4^N$, for the initial chain we have $5>4$.

\subsection{The graph $\Gamma_N$}
We encode the complete binary tree of height $N$ by its \emph{contour word}
$$
 \ell_0 = \ta,\quad \ell_{h+1} = \tx\,\ell_h\,\ell_h\,\ty,\quad
 L_h := |\ell_h| = 3\cdot2^h-2 .
$$
The symbols $\tx$ (``down'') and $\ty$ (``up'') are auxiliary.
The positions between consecutive letters of $\ell_N$, including
both ends, are called \emph{gaps} and are numbered $0,1,\ldots,L_N$.
The
\emph{height} $\eta(p)$ of a gap $p$ is the number of letters $\tx$ minus the number
of letters $\ty$ to the left of $p$.
Every occurrence of a subword $\ell_h$ produced
by the recursion is a \emph{vertex} of level $N-h$.
It occupies an interval of gaps
$[p,q]$ with $\eta(p) = \eta(q)$ equal to its level.
The root is $[0,L_N]$; the
leaves are the letters $\ta$ and have level $N$.
An internal vertex $[p,q]$ has the
form $\tx\,\ell\,\ell\,\ty$, and its children are $[p+1,m]$ and $[m,q-1]$ for a~suitable~$m$.
Two vertices are either nested or have disjoint interiors, and every gap strictly
inside a vertex of level~$j$ has height greater than~$j$.

\begin{definition}
$\Gamma_N$ is the directed graph whose vertices are the gaps $0,\ldots,L_N$ and whose
labelled edges are:
\begin{itemize}
\item for each letter $\tx$ or $\ty$ of $\ell_N$, five parallel edges $p\to p+1$
across it, labelled $z_1,\ldots,z_5$;
\item for each vertex $[p,q]$ of the tree, an edge $p\to q$ labelled $a$, the
\emph{jump} over this vertex.
For a leaf this is the edge across its own letter $\ta$.
\end{itemize}
\end{definition}

\begin{figure}[ht]
\centering
\begin{tikzpicture}[x=11mm,y=-10mm,>={Stealth[length=2mm]},
  dot/.style={circle,fill,inner sep=1.3pt}]
\def\H{{0,1,2,2,2,1,2,2,2,1,0}}
\foreach \p in {0,...,10}{\pgfmathsetmacro\h{\H[\p]}
  \node[dot] (g\p) at (\p,\h) {}; \node[above=1pt,font=\scriptsize] at (\p,\h) {$\p$};}
\foreach \a/\b in {0/1,1/2,5/6}{\draw[->,thick] (g\a)--node[left,font=\tiny,pos=0.45]{$\tx$}(g\b);}
\foreach \a/\b in {4/5,8/9,9/10}{\draw[->,thick] (g\a)--node[right,font=\tiny,pos=0.55]{$\ty$}(g\b);}
\foreach \a/\b in {2/3,3/4,6/7,7/8}{\draw[->,dashed] (g\a)--(g\b);}
\draw[->] (g1) .. controls +(0.6,2.2) and +(-0.6,2.2) .. node[below,font=\scriptsize]{$a$} (g5);
\draw[->] (g5) .. controls +(0.6,2.2) and +(-0.6,2.2) .. node[below,font=\scriptsize]{$a$} (g9);
\draw[->] (g0) .. controls +(2.5,-1.3) and +(-2.5,-1.3) .. node[above,font=\scriptsize]{$a$} (g10);
\node[font=\scriptsize,right] at (10.4,0) {level 0};
\node[font=\scriptsize,right] at (10.4,1) {level 1};
\node[font=\scriptsize,right] at (10.4,2) {level 2};
\end{tikzpicture}
\caption{The graph $\Gamma_2$ for $\ell_2 = \tx\tx\ta\ta\ty\tx\ta\ta\ty\ty$.
Short
solid arrows are bundles of five edges $z_1,\ldots,z_5$ (across the letters $\tx$
and $\ty$); dashed arrows are the leaf edges $a$; arcs are the jumps over the
internal vertices $[1,5]$, $[5,9]$ and $[0,10]$.}
\label{fig:gamma}
\end{figure}
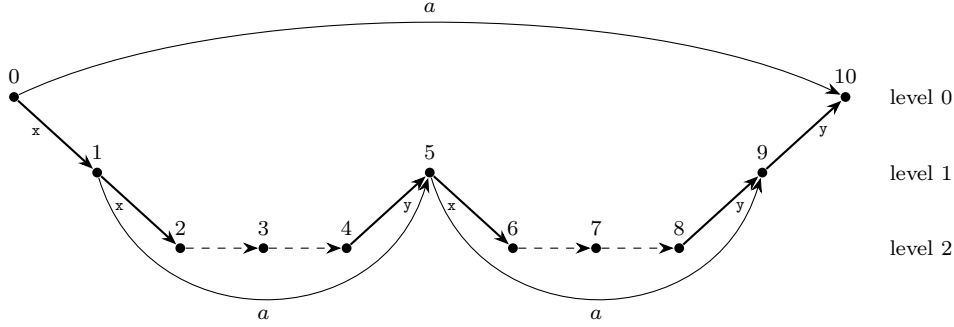

We will use the following properties of $\Gamma_N$:

\noindent(G1) Every edge goes forward: $p\to q$ with $q>p$;

\noindent(G2) Edges $z_i$ change the height by $\pm1$; jumps preserve it.
An edge from
$p$ skips only gaps strictly inside a vertex starting at $p$, so all skipped gaps have height $h>\eta(p)$;

\noindent(G3) From each gap there is at most one edge with a given label.

By (G3), a path is determined by its initial gap and its
word.

\subsection{The representation $T_S$}

Let $S$ be a chain of height $N$.
To a gap $p$ we attach the space
$K^{d_{\eta(p)}}$, and we put $V_S = \bigoplus_{p = 0}^{L_N}K^{d_{\eta(p)}}$, so that
endomorphisms of $V_S$ are block matrices indexed by pairs of gaps.
The
\emph{weight} of an edge $z_i$ going from level $j$ to level $j+1$ is $X_{j,i}$;
that of an edge $z_i$ going from level $j+1$ to level~$j$ is $Y_{j,i}$; the weight of a jump is the identity matrix.
For a~letter $c\in\{a,z_1,\ldots,z_5\}$, the
$(p,q)$ block of $T_S(c)$ is the weight of the edge $p\to q$ labelled~$c$, or~0 if there is no such edge.
We index blocks by (initial gap, final gap); with this convention,
block multiplication concatenates directed paths from left to right.
This defines a homomorphism
$T_S\colon A\to\End(V_S)$.
We use the same notation for its canonical unital extension
$A^\#\to\End(V_S)$ when the empty word is allowed.
By the rule for multiplying block matrices,
the $(p,q)$ block of $T_S(w)$ for a~word~$w$, is the weight of the path from~$p$ to $q$ with word~$w$ (i.\,e. the product of the weights of its edges in
order), or 0~if there is no such path.

\begin{lemma}\label{defect}
a) A path with word $aa$ consists of the jumps over two siblings, from the start
of the first to the end of the second.
There are no paths with word $aaa$; hence
$T_S(a)^3 = 0$.

b) A path with word $z_iaaz_j$ enters an internal vertex $[p,q]$, jumps over both
of its children and leaves it; thus it goes from $p$ to $q$.

c) The root block $T_S(a)_{0,L_N}$ equals $I_D$, and the only nonzero blocks of
$T_S(\omega)$ are the entries $1$ at the $2^N$ leaf edges.
Consequently,
$\rk T_S(a)\ge D$ and $\rk T_S(\omega) = 2^N$.
\end{lemma}

\begin{proof}
(a) Two consecutive $a$-jumps can occur only over two siblings: the jump over the first child ends at the initial gap of the second, whereas after the second child no vertex starts.
Hence there are no paths with word $aaa$.

(b) By (a), the middle part $aa$ jumps over the two children of an internal vertex $P=[p,q]$.
The two $z$-edges necessarily cross the initial $\tx$ and the terminal $\ty$ of~$P$, so the whole path goes from $p$ to~$q$.

(c) The nonzero blocks of $T_S(a)$ are exactly the jumps.
For an internal vertex $[p,q]$ of level~$t$, part~(b) gives the corresponding block
$\sum_iX_{t,i}Y_{t,i}=I_{d_t}$ of $\sum_iT_S(z_ia^2z_i)$, which cancels the jump block $T_S(a)_{p,q}$.
Thus $T_S(\omega)$ consists exactly of the $2^N$ leaf blocks $1=I_{d_N}$; they lie in distinct block rows and columns, so $\rk T_S(\omega)=2^N$.
The root jump is the block $I_D$, hence $\rk T_S(a)\ge D$.
\end{proof}

Thus, the tree separates the two quantities that will matter in Proposition 3.2: the root contributes $D$ independent directions to $T_S(a)$, whereas $T_S(\omega)$ sees only the $2^N$ leaves.
From now on the only task is to make prescribed elements nilpotent without destroying the inequality $D>4^N$.

\section{Crossings and the cut tree}\label{s:passes}

Let $S$ be a chain of height $N$, and fix a~divisor~$n$ of~$N$.
The heights~$L$ divisible by~$n$ are called \emph{grid heights},
and gaps of grid height are \emph{grid gaps}.

\begin{definition}
A \emph{crossing} of a path is a segment of it between two consecutive visits to grid
gaps whose heights are different.
\end{definition}
By (G2), the endpoints of a crossing have heights~$L$ and $L+n$ (in one of the two orders) for some grid height~$L\in \{0,n,2n,\ldots,N-n\}$, and all its interior gaps have heights strictly between $L$ and $L+n$.
A crossing has at least $n$~edges.
Crossings of a path do not overlap.

\begin{lemma}\label{long}
Every path in $\Gamma_N$ with at least $2^{2n+1}$ edges contains a crossing.
\end{lemma}

\begin{proof}
Note that a path whose gap heights lie in an interval
$[b,b+R]$ has at most $2L_R$ edges.
If $b=0$, it uses at most $L_R+1$ gaps of height at most $R$.
If $b>0$, it stays inside the two children of a vertex of level $b-1$:
leaving them would force a visit to height $b-1$.
Each child contains at most $L_R+1$ gaps of relative height at most~$R$,
and the two children share one gap.
Indeed, truncating a subtree at relative height~$R$ gives $\ell_{\min(h,R)}$, hence at most $L_R+1$ such gaps.
Thus, in either case the path uses at most $2L_R+1$ gaps.
By~(G1), no gap is visited twice, so it has at most $2L_R$ edges.

Now suppose that the path has no crossing.
If it visits grid gaps, all grid gaps visited have the same height~$g$,
so all its heights lie in $[g-n+1,g+n-1]$.
If it visits no grid gap, all its heights lie strictly between two consecutive grid heights.
Hence in either case the heights lie in an interval of width at most~$2n-2$, and therefore the path has at most
$2L_{2n-2} = 6\cdot2^{2n-2}-4  < 2^{2n+1}$ edges.
\end{proof}

\subsection{The cut tree}\label{s:cut}
Let $L$ be a grid height with $L+n\le N$.
The graph $\Gh_n$ is obtained from $\Gamma_n$ by deleting its $2^n$ leaf edges $a$; the jumps over internal vertices are kept.
Spaces and weights are taken from the chain shifted by~$L$: to a gap~$p$ we attach $K^{d_{L+\eta(p)}}$, and the edges get the matrices $X_{L+j,i}$, $Y_{L+j,i}$.
This gives a~homomorphism
$$
\Th_{S,L}\colon A \to\End\Bigl(\mathop{\bigoplus}\limits_p K^{d_{L+\eta(p)}}\Bigr)
$$
with the same rule ``block $ = $ weight of path''.
When empty words are allowed, we use the same notation for the
canonical unital extension to $A^\#$.

Each of the $2^{n-1}$ lowest internal vertices of $\Gamma_n$ (a subword
$\tx\ta\ta\ty$) had two leaf edges.
After their deletion, its first bottom gap has
incoming but no outgoing edges, its last bottom gap has outgoing but no incoming edges, and its middle gap is isolated.
The first ones are called \emph{entry gaps}, the last ones \emph{exit gaps}.
A \emph{boundary pair} is a pair $(0,g')$ with $g'$ an
entry gap, or a pair $(g,L_n)$ with $g$ an exit gap.
There are $2^{n-1}+2^{n-1} = 2^n$ boundary pairs.

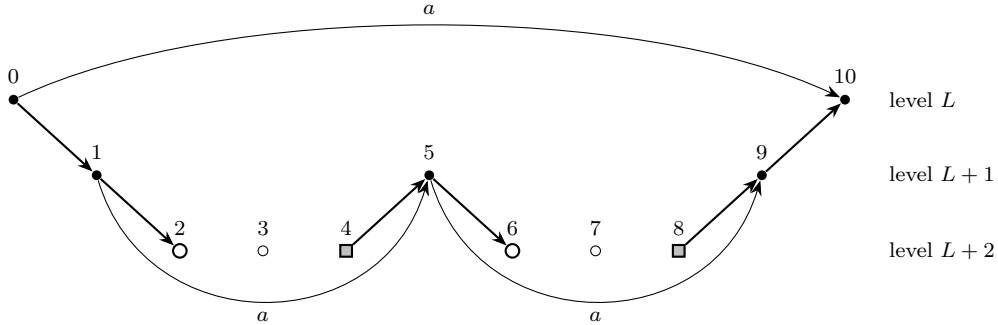
\begin{figure}[ht]
\centering
\begin{tikzpicture}[x=11mm,y=-10mm,>={Stealth[length=2mm]},
  dot/.style={circle,fill,inner sep=1.3pt},
  ent/.style={circle,draw,fill=white,inner sep=1.8pt,thick},
  ext/.style={rectangle,draw,fill=black!25,inner sep=2.2pt,thick},
  iso/.style={circle,draw,inner sep=1.3pt}]
\def\H{{0,1,2,2,2,1,2,2,2,1,0}}
\foreach \p in {0,1,5,9,10}{\pgfmathsetmacro\h{\H[\p]}\node[dot] (g\p) at (\p,\h) {};}
\foreach \p in {2,6}{\node[ent] (g\p) at (\p,2) {};}
\foreach \p in {4,8}{\node[ext] (g\p) at (\p,2) {};}
\foreach \p in {3,7}{\node[iso] (g\p) at (\p,2) {};}
\foreach \p in {0,...,10}{\pgfmathsetmacro\h{\H[\p]}\node[above=2pt,font=\scriptsize] at (\p,\h) {$\p$};}
\foreach \a/\b in {0/1,1/2,5/6}{\draw[->,thick] (g\a)--(g\b);}
\foreach \a/\b in {4/5,8/9,9/10}{\draw[->,thick] (g\a)--(g\b);}
\draw[->] (g1) .. controls +(0.6,2.2) and +(-0.6,2.2) .. node[below,font=\scriptsize]{$a$} (g5);
\draw[->] (g5) .. controls +(0.6,2.2) and +(-0.6,2.2) .. node[below,font=\scriptsize]{$a$} (g9);
\draw[->] (g0) .. controls +(2.5,-1.3) and +(-2.5,-1.3) .. node[above,font=\scriptsize]{$a$} (g10);
\node[font=\scriptsize,right] at (10.4,0) {level $L$};
\node[font=\scriptsize,right] at (10.4,1) {level $L+1$};
\node[font=\scriptsize,right] at (10.4,2) {level $L+2$};
\end{tikzpicture}
\caption{The cut graph $\Gh_2$.
The leaf edges are removed.
White circles are the entry gaps $2,6$, grey squares the exit gaps $4,8$, and the gaps $3,7$ are isolated.
The boundary pairs are $(0,2)$, $(0,6)$, $(4,10)$, $(8,10)$.
For example, the path $0\to1\to2$ (word $z_i z_j$) is a~crossing from~0 to the entry gap~2. The path
$0\to1\xrightarrow{a}5\to6$ is a crossing to the entry gap~6 that jumps over the first vertex of level $L+1$.}
\label{fig:cut}
\end{figure}

\begin{lemma}\label{cut}
Every crossing between heights $L$ and $L+n$ lies in a~unique vertex~$v$ of level~$L$.
Identify the gaps of $\Gh_n$ with the gaps of $v$ that do not lie strictly inside vertices of level $L+n$.
Then the crossings between $L$ and $L+n$ lying in $v$ are
exactly the paths of $\Gh_n$ between boundary pairs.
The correspondence preserves words and weights.
\end{lemma}

\begin{proof}
Consider first a crossing from height~$L$ to height~$L+n$.
Its first edge is a $z_i$-edge increasing the height, hence it crosses the initial~$\tx$ of a~unique vertex~$v$ of level~$L$.
The crossing cannot leave~$v$, since this would require another visit to height~$L$, and it cannot enter the interior of a~vertex of level $L+n$, since it would first have to visit its initial gap, already of height $L+n$.
Thus, after identifying the gaps of~$\Gh_n$ with the gaps of~$v$ that do not lie strictly inside vertices of level $L+n$, the crossing is a~path from~0 to an entry gap.
A~crossing from height $L+n$ to height $L$ is treated in the same way and gives a path from an exit gap to $L_n$.

For $L>0$ the local geometry is illustrated below; for $L=0$, $v$~is the root.

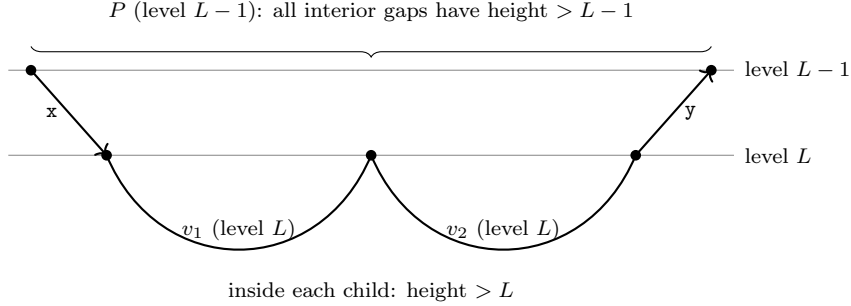
\begin{center}
\begin{tikzpicture}[x=10mm,y=-9mm]
\draw[black!45] (-0.3,0)--(9.3,0) node[right,black,font=\scriptsize]{level $L-1$};
\draw[black!45] (-0.3,1.25)--(9.3,1.25) node[right,black,font=\scriptsize]{level $L$};
\fill (0,0) circle (2pt); \fill (9,0) circle (2pt);
\fill (1,1.25) circle (2pt); \fill (4.5,1.25) circle (2pt); \fill (8,1.25) circle (2pt);
\draw[thick,->] (0,0)--node[left,font=\scriptsize]{$\tx$}(1,1.25);
\draw[thick] (1,1.25) .. controls (1.7,3.1) and (3.8,3.1) .. (4.5,1.25);
\draw[thick] (4.5,1.25) .. controls (5.2,3.1) and (7.3,3.1) .. (8,1.25);
\draw[thick,->] (8,1.25)--node[right,font=\scriptsize]{$\ty$}(9,0);
\node[font=\scriptsize] at (2.75,2.35) {$v_1$ (level $L$)};
\node[font=\scriptsize] at (6.25,2.35) {$v_2$ (level $L$)};
\node[font=\scriptsize] at (4.5,3.25) {inside each child: height $>L$};
\draw[decorate,decoration={brace,amplitude=4pt,mirror}] (0,-0.35)--(9,-0.35)
 node[midway,above=5pt,font=\scriptsize]{$P$ (level $L-1$): all interior gaps have height $>L-1$};
\end{tikzpicture}
\end{center}

Conversely, let a path of $\Gh_n$ go from~0 to an entry gap $g'$.
It has no interior gaps at height~0, where the only gaps are~0 and~$L_n$.
It has none at height~$n$ either: no edge leaves an entry gap, no edge enters an exit gap, and the middle gaps are isolated.
All edges of~$\Gh_n$ are edges of~$\Gamma_N$ inside~$v$ with the same weights.
Hence the path is a~crossing in $\Gamma_N$.
Pairs $(g,L_n)$ are treated analogously.
\end{proof}

\subsection{Boundary conditions}
Let $W_d$ be the set of all words of length at most~$d$ in the letters
$a,z_1,\ldots,z_5$, including the empty word, and let
$s(d) = |W_d| = 1+6+\ldots+6^d$.

The conditions below are chosen to annihilate all words that can occur on a first crossing in Section~\ref{s:first}.

\begin{definition}\label{conditions}
Let $0\ne f\in A$, $d = \deg f<n$ and $n\mid N$.
A chain $S$ of height $N$
\emph{satisfies the conditions $(f,n)$} if for every $L$ such that $0\leq L\leq N-n$ and $n\mid L$, every
boundary pair $(g,g')$, all $u,v\in W_d$ and all $0\le r\le L_n$,
\begin{equation}\label{BC}
 \Th_{S,L}(u\,f^r\,v)_{g,g'} = 0 .
\end{equation}
\end{definition}
For $r>L_n$ the equalities~\eqref{BC} hold automatically: every monomial of $uf^rv$ has more than $L_n$ letters, since the monomials of~$f$ are nonempty, while
by (G1) a~path in $\Gh_n$ has at most $L_n$ edges.

\section{The first-crossing decomposition}\label{s:first}

Let $S$ be a chain of height~$N$, let $n\mid N$, and let
$f = \sum_{w\in\supp f}c_w\,w\in A$, $d = \deg f<n$, and $e\ge 2^{2n+1}$.
We first display the splitting used throughout this section.
For a contributing word,
once its first crossing begins inside $w_\alpha$ and ends inside $w_\beta$, write
$w_\alpha = u'u$ and $w_\beta = vv'$, and write down the whole word as follows:
$$
w_1\ldots w_e
 = \underbrace{w_1\ldots w_{\alpha-1}u'}_{\text{prefix}}
   \underbrace{u\,w_{\alpha+1}\ldots w_{\beta-1}v}_{\text{first crossing}}
   \underbrace{v'w_{\beta+1}\ldots w_e}_{\text{suffix}}.
$$
Expanding the power,
$$
T_S(f^e)_{p,q} = \sum_{(w_1,\ldots,w_e)}c_{w_1}\ldots c_{w_e}T_S(w_1\ldots w_e)_{p,q},
$$
where the sum runs over all sequences of monomials $w_k\in\supp f$.
A sequence contributes only if there is a~path~$\gamma$ from~$p$ to~$q$ with word $w_1\ldots w_e$; this path is unique by (G3).
It has at least $e$~edges, so by Lemma~\ref{long} it contains a crossing.
Let $\gamma[g\to g']$ be its first crossing.

We attach to the sequence the following data:
\begin{itemize}
\item the indices $\alpha\le\beta$ of the monomials containing the beginning and
the end of the crossing (a beginning at a junction of two monomials is assigned to the right one, an end to the left one);
\item the factorizations $w_\alpha = u'u$ and $w_\beta = vv'$ at the points~$g$ and~$g'$, so that $0\leq|u'|<|w_\alpha|$ and $0<|v|\le|w_\beta|$;
\item all monomials $w_k$ with $k\notin(\alpha,\beta)$;
\item the gap $g'$.
\end{itemize}
The crossing has at least $n>d\ge|w_k|$ edges, so $\alpha<\beta$, and its word is
$u\,w_{\alpha+1}\ldots w_{\beta-1}\,v$, where $u,v\in W_d$ and $r = \beta-\alpha-1\geq0$.
By Lemma~\ref{cut} and (G1) the crossing has at most $L_n$ edges, and each monomial contributes at least one edge, so $r\leq L_n$.
These data determine the prefix and suffix, the gap~$g$, the grid height
$L=\min(\eta(g),\eta(g'))$, and the boundary pair $(\hat g,\hat g')$ corresponding to $(g,g')$ under Lemma~\ref{cut}.
Fig.~\ref{fig:first} illustrates the decomposition.

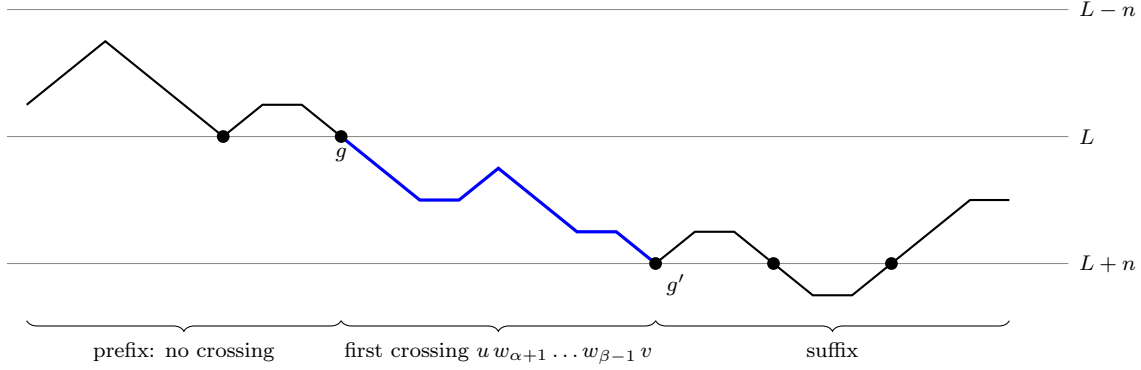
\begin{figure}[ht]
\centering
\begin{tikzpicture}[x=5.2mm,y=4.2mm,>={Stealth[length=1.8mm]}]
\draw[black!45] (-0.5,0)--(26.5,0) node[right,black,font=\scriptsize]{$L+n$};
\draw[black!45] (-0.5,4)--(26.5,4) node[right,black,font=\scriptsize]{$L$};
\draw[black!45] (-0.5,8)--(26.5,8) node[right,black,font=\scriptsize]{$L-n$};
\draw[thick] (0,5)--(1,6)--(2,7)--(3,6)--(4,5)--(5,4)--(6,5)--(7,5)--(8,4);
\draw[very thick,blue] (8,4)--(9,3)--(10,2)--(11,2)--(12,3)--(13,2)--(14,1)--(15,1)--(16,0);
\draw[thick] (16,0)--(17,1)--(18,1)--(19,0)--(20,-1)--(21,-1)--(22,0)--(23,1)--(24,2)--(25,2);
\foreach \x/\y in {5/4,8/4,16/0,19/0,22/0}{\fill (\x,\y) circle (2.4pt);}
\node[below,font=\scriptsize] at (8,4) {$g$};
\node[below right,font=\scriptsize] at (16,0) {$g'$};
\draw[decorate,decoration={brace,amplitude=4pt,mirror}] (0,-1.8)--(8,-1.8)
  node[midway,below=4pt,font=\scriptsize]{prefix: no crossing};
\draw[decorate,decoration={brace,amplitude=4pt,mirror}] (8,-1.8)--(16,-1.8)
  node[midway,below=4pt,font=\scriptsize]{first crossing $u\,w_{\alpha+1}\ldots w_{\beta-1}\,v$};
\draw[decorate,decoration={brace,amplitude=4pt,mirror}] (16,-1.8)--(25,-1.8)
  node[midway,below=4pt,font=\scriptsize]{suffix};
\end{tikzpicture}
\caption{The first-crossing decomposition (heights are plotted downwards, as in the
tree).
Dots mark visits to grid gaps.
The prefix visits grid gaps of one height only.
The first crossing (thick, blue) runs from $g$ to $g'$ between two neighbouring grid heights.}
\label{fig:first}
\end{figure}

\begin{lemma}\label{factor}
We have
$$
T_S(f^e)_{p,q} = \sum_{\text{data}}
 \Bigl(\prod_{k\in\{1,\ldots,e\}\setminus \{\alpha+1,\ldots,\beta-1\}}c_{w_k}\Bigr)\,
 \wt(\text{prefix})\cdot
 \Th_{S,L}\bigl(u\,f^{r}\,v\bigr)_{\hat g,\hat g'}\cdot
 \wt(\text{suffix}),
$$
where the sum runs over all data that occur for at least one contributing sequence;
here $u,v\in W_d$, $0\leq r\leq L_n$, and $(\hat g,\hat g')$ is a boundary pair.
\end{lemma}

\begin{proof}
Every contributing sequence has exactly one set of data.
Fix data that occur and sum over the inner monomials $w_{\alpha+1},\ldots,w_{\beta-1}$.
We claim that the
sequences with these data are exactly those for which
$u\,w_{\alpha+1}\ldots w_{\beta-1}\,v$ is the word of a~path in $\Gh_n$ from $\hat g$ to~$\hat g'$.

If a sequence has these data, its first crossing gives such a path by Lemma~\ref{cut}.
Conversely, suppose that such a path exists.
By Lemma~\ref{cut} it corresponds to a crossing from~$g$ to $g'$.
Attaching the fixed prefix and suffix gives a path from~$p$ to~$q$ with word $w_1\ldots w_e$; since the prefix has no crossing, $g\to g'$ is its first crossing.
Hence the data of this sequence are the given ones.

The weight of the whole path is the product of the weights of the prefix, the crossing and the suffix, and the weight of the crossing equals the weight of the corresponding
path in $\Gh_n$.
Summing these weights with the coefficients
$c_{w_{\alpha+1}}\ldots c_{w_{\beta-1}}$ gives exactly the block
$\Th_{S,L}(uf^rv)_{\hat g,\hat g'}$.
The words $u$ and $v$ are fixed fragments; only the~$r$ complete copies of~$f$ contribute coefficients inside the middle block.
\end{proof}

\begin{theorem}\label{nilpotence}
If a chain $S$ of height $N$ satisfies the conditions $(f,n)$, then $T_S(f)^{2^{2n+1}} {=} 0$.
\end{theorem}

\begin{proof}
By Lemma~\ref{factor} with $e = 2^{2n+1}$, every block of $T_S(f)^e = T_S(f^e)$ is a
sum of terms whose middle factor has the form~\eqref{BC}.
All these factors vanish.
\end{proof}

\section{Amplification and compression}\label{s:operations}

Compression will impose the new boundary conditions losing at most $C(f)4^n$ root dimensions.
Amplification replaces $(N,D)$ by $(kN,D^k)$ preserving all old conditions.

\begin{lemma}\label{amplification}
For every $k\geq1$, a~chain of height $N$ with root dimension $D$ can be replaced by a~chain of height $kN$ with root dimension $D^k$ so that all conditions $(f,n)$ with $n\mid N$ are
preserved.
\end{lemma}

\begin{proof}
Put $k$ copies of the chain one after another and tensor the copy number
$t = 0,\ldots,k-1$ with $I_{D^{k-t-1}}$:
$$
 d'_{tN+j} = d_jD^{k-t-1},\quad
 X'_{tN+j,i} = X_{j,i}\otimes I_{D^{k-t-1}},\quad
 Y'_{tN+j,i} = Y_{j,i}\otimes I_{D^{k-t-1}},\quad 0\leq j<N,
$$
and $d'_{kN} = 1$; see Fig.~\ref{fig:amp}.
At the junctions,
$d_ND^{k-t} = D^{k-t}$ and $d_0D^{k-t-1} = D^{k-t}$, so the dimensions agree.
The identities $\sum_iX'Y' = I$ follow from the
old ones by tensoring.
Since $n\mid N$, every band $[L,L+n]$ lies in one copy, and there $\Th_{S',L}(h) = \Th_{S,L'}(h)\otimes I$ blockwise, where $L'\equiv L\pmod N$ and $n\mid L'$.
Zero blocks remain zero.
\end{proof}

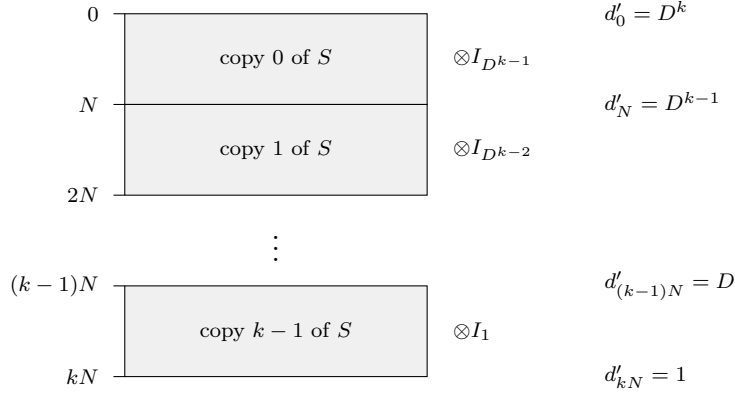
\begin{figure}[ht]
\centering
\begin{tikzpicture}[x=10mm,y=-6mm]
\foreach \t/\name/\lab in {0/{$0$}/{$\otimes I_{D^{k-1}}$},1/{$1$}/{$\otimes I_{D^{k-2}}$},3/{$k-1$}/{$\otimes I_{1}$}}{
  \draw[fill=black!6] (0,{2*\t}) rectangle (4,{2*\t+2});
  \node[font=\scriptsize] at (2,{2*\t+1}) {copy \name\ of $S$};
  \node[font=\scriptsize,right] at (4.2,{2*\t+1}) {\lab};}
\node at (2,5) {$\vdots$};
\foreach \y/\lab in {0/{$0$},2/{$N$},4/{$2N$},6/{$(k-1)N$},8/{$kN$}}{
  \draw (-0.15,\y)--(0,\y); \node[left,font=\scriptsize] at (-0.2,\y) {\lab};}
\node[font=\scriptsize,right] at (6.2,0) {$d'_0=D^k$};
\node[font=\scriptsize,right] at (6.2,2) {$d'_N=D^{k-1}$};
\node[font=\scriptsize,right] at (6.2,6) {$d'_{(k-1)N}=D$};
\node[font=\scriptsize,right] at (6.2,8) {$d'_{kN}=1$};
\end{tikzpicture}
\caption{Amplification: $k$ copies of a chain stacked on top of each other.
Each band
of height $n\mid N$ lies inside one copy, where all matrices are tensored with an
identity matrix.}
\label{fig:amp}
\end{figure}

In the compression step below, where $N=n$ and hence $d_n=1$,
we will need to make outgoing boundary blocks, which are vectors in $K^D$, vanish, and incoming boundary blocks, which are linear functionals on $K^D$, vanish.
The next lemma does both simultaneously at a loss of at most one dimension per condition.

\begin{lemma}\label{linalg}
Let $v_1,\ldots,v_p\in K^D$ and let $\varphi_1,\ldots,\varphi_q\in(K^D)^*$.
Then there are an integer $D'\ge D-p-q$ and linear maps $\iota\colon K^{D'}\to K^D$ and
$\pi\colon K^D\to K^{D'}$ such that
$$
\pi\iota = I_{D'},\quad \pi(v_i) = 0,\ 1\leq i\leq p,\quad
\varphi_j\iota = 0,\ 1\leq j\leq q.
$$
\end{lemma}

\begin{proof}
Put $U = \Span\{v_1,\ldots,v_p\}$ and
$H = \bigcap_{j = 1}^q\ker\varphi_j$.
Then $\dim U\le p$ and $\dim H\ge D-q$.
Choose a complement $W$ of $H\cap U$ in~$H$.
Thus $W\cap U = (0)$ and $\dim W\ge D-q-p$.
Identify~$W$ with $K^{D'}$, where $D' = \dim W$, let
$\iota\colon K^{D'}\to K^D$ be the inclusion, and let
$\pi\colon K^D\to K^{D'}$ be the projection onto~$W$ along a~complement of~$W$ containing~$U$.
Then $\pi\iota = I_{D'}$, every $v_i$ lies in $\ker\pi$, and
$\iota(K^{D'}) = W\subseteq H\subseteq\ker\varphi_j$ for every $j$.
\end{proof}

\begin{lemma}\label{compression}
Let $0\ne f\in A$ and put $d = \deg f$.

\noindent a) Let $S$ be a chain of height $N$ and root dimension $D$, and $D'>0$.
If $\iota\colon K^{D'}\to K^D$ and $\pi\colon K^D\to K^{D'}$ satisfy $\pi\iota = I_{D'}$, then the replacement
$$
d_0\to D',\quad X_{0,i}\to\pi X_{0,i},\quad
Y_{0,i}\to Y_{0,i}\iota
$$
produces a chain.
All other dimensions and matrices remain unchanged, and every condition satisfied by~$S$ remains satisfied.

\noindent b) Suppose that $S$ has height $n>d$ and root dimension $D>C(f)4^n$, where $C(f) = 3s(d)^2$.
Then a~suitable replacement as in a) yields a~chain that
satisfies the conditions $(f,n)$ and has root dimension $D'$ such that $D'>D-C(f)4^n$.
\end{lemma}

\begin{proof}
For a),
$\sum_i\pi X_{0,i}Y_{0,i}\iota = \pi\iota = I_{D'}$.
Bands with $L>0$ do not change.
For $L = 0$, the path of a boundary pair $(0,g')$ touches the root level only at its first edge, and the path of a~boundary pair $(g,L_n)$ only at its last edge.
Hence these blocks are multiplied by $\pi$ on the left or by $\iota$ on the right, respectively, and zero blocks remain zero.

For b), $N = n$, so there is a~single band $L = 0$, and $d_n = 1$.
Denote the column blocks~\eqref{BC} belonging to the boundary pairs $(0,g')$ by $v_1,\ldots,v_p\in K^D$, and denote the row blocks belonging to the boundary pairs
$(g,L_n)$ by $\varphi_1,\ldots,\varphi_q\in(K^D)^*$.
Each triple $(u,v,r)$ gives $2^{n-1}$ columns and $2^{n-1}$ rows.
Therefore
$$
p+q \leq 2^n s(d)^2(L_n+1)
    <3s(d)^2 4^n\\
    = C(f)4^n,
$$
since $L_n+1<3\cdot2^n$.

We apply Lemma~\ref{linalg} to these vectors $v_i$ and
functionals $\varphi_j$.
By a), the new blocks are $\pi(v_i) = 0$ and
$\varphi_j\iota = 0$.
Moreover, $D'\geq D-p-q>D-C(f)4^n>0$.
\end{proof}

\section{The construction}\label{s:construction}

Let $K$ be a countable field and $f_1,f_2,\ldots$ an enumeration of the nonzero elements of~$A$.
We construct chains $S_i$ of heights $n_i$ with root dimensions $D_i$ such that
$$
D_i>4^{n_i},\quad S_i\text{ satisfies }(f_j,n_j)\text{ and }n_j\mid n_i\text{ for all }1\leq j\leq i,
$$
this is the induction invariant.

Let $S_0$ be the initial chain: $n_0 = 1$ and $D_0 = 5>4$.
Suppose that $S_{i-1}$ of height~$N$ and root dimension $D>4^N$ has been constructed.
Choose an integer $k\geq 2$ such that
\begin{equation}\label{choice}
kN>\deg f_i,\quad (D/4^N)^k>1+C(f_i);
\end{equation}
we can do this since $D/4^N>1$.
Amplify $S_{i-1}$ by the factor $k$ by Lemma~\ref{amplification}.
This gives height $kN$ and root dimension
$D^k>(1+C(f_i))4^{kN}$.
Then compress the result adding the conditions $(f_i,kN)$
by Lemma~\ref{compression}.
The new chain $S_i$ has height $n_i = kN$ and root dimension
$$
D_i > D^k-C(f_i)4^{kN}
    > 4^{kN}
    = 4^{n_i}.
$$
The old conditions $(f_j,n_j)$, $j<i$, remain valid under both operations, since every $n_j$ divides~$N$; and every $n_j$ divides the new height.
Moreover $n_i\geq 2^i$, so $n_i\to\infty$.

For every $j$ put $q_j = 2^{2n_j+1}$.
For $i\geq j$ the chain $S_i$ satisfies $(f_j,n_j)$
with $n_j\mid n_i$, so $T_{S_i}(f_j)^{q_j} = 0$ by Theorem~\ref{nilpotence}.
By
Lemma~\ref{defect}, $\rk T_{S_i}(a)\ge D_i>4^{n_i}$ and
$\rk T_{S_i}(\omega) = 2^{n_i}$.
Thus the representations $T_i = T_{S_i}$ satisfy all
the requirements~\eqref{hyp} of Proposition~\ref{fromMatrices}.

\begin{proof}[Proof of Theorem~\ref{main}]
By Proposition~\ref{fromMatrices}, the representations $T_{S_i}$ yield a nil algebra~$R$ with an element
$0\neq b = \sum_jc_jb^2c_j\in Rb^2R$.
By Proposition~\ref{ryab}, we get a~simple nil ring.
\end{proof}

\noindent Vsevolod Gubarev \\
Sobolev Institute of Mathematics \\
Acad. Koptyug ave. 4, 630090 Novosibirsk, Russia \\
Novosibirsk State University \\
Pirogova str. 1, 630090 Novosibirsk, Russia \\
e-mail: wsewolod89@gmail.com
\end{document}